\documentclass[a4paper, leqno]{amsart}

\usepackage{amsmath}
\usepackage{amssymb}
\usepackage{amsfonts}
\usepackage{amsthm}
\usepackage{amscd}
\usepackage{enumerate}
\usepackage[mathscr]{eucal}

\newtheorem{theo}{Theorem}
\newtheorem{lemm}[theo]{Lemma}
\newtheorem{coro}[theo]{Corollary}
\newtheorem{prop}[theo]{Proposition}

\newcommand{\Int}{\mathbb{Z}}
\newcommand{\Rat}{\mathbb{Q}}
\newcommand{\Real}{\mathbb{R}}
\newcommand{\CP}{\mathbb{C}P}
\newcommand{\Res}{\mathrm{Res}}
\newcommand{\ord}{\mathrm{ord}}
\newcommand{\cmplx}{\mathbb{C}}

\begin{document}
\title[First Pontrjagin 
class of homotopy CP]{The first Pontrjagin
classes of \\ homotopy complex projective spaces}
\author[Y.~Kitada]{Yasuhiko Kitada}
\address[Y.~Kitada]{Yokohama National University}
\email[Y.~Kitada]{ykitada@ynu.ac.jp}
\author[M.~Nagura]{Maki Nagura}
\address[M.~Nagura]{Yokohama National University}
\email[M.~Nagura]{maki@ynu.ac.jp}
\maketitle
\begin{abstract}
Let $M^{2n}$ be a closed smooth manifold homotopy equivalent to
the complex projective space $\CP(n)$.  The purpose of this 
paper is to show that when $n$ is even, the difference
of the first Pontrjagin classes between $M^{2n}$ and $\CP(n)$
is divisible by 16.
\end{abstract}
\section{Introduction and the main theorem }
Let $M^{2n}$ be a closed smooth manifold and $\CP(n)$ be
the complex projective space of complex dimension $n$. 
If there is a homotopy equivalence $ f : M^{2n}\rightarrow \CP(n)$,
we say that $M^{2n}$ is a homotopy projective space or more briefly
a homotopy $\CP(n)$.  
When $M^{2n}$ is a homotopy projective space and
 $ f : M^{2n}\rightarrow \CP(n)$ is a homotopy equivalence, 
 define an integer $\delta(M)$ by
$p_1(M)-f^{*}(p_1(\CP(n)))=\delta(M) u^2$, 
where $p_1(M)$ is the first
Pontrjagin class of $M^{2n}$ and $u$ is a generator of
$H^2(M^{2n}; \Int)$.  Clearly $\delta(\CP(n))$ is zero and
$\delta(M)$ measures the difference between the first Pontrjagin classes
of the two manifolds $M^{2n}$ and $\CP(n)$.
When $n=3$,  Montgomery and Yang studied
and classified homotopy complex projective spaces $M^6$ 
and proved that the first Pontrjagin class $p_1(M)$ 
is of the form $p_1(M^6)=(4+24\alpha(M))u^2$
for some integer $\alpha(M)$ where $u\in H^2(M^6;\Int)$ 
is a generator.
They also showed that the diffeomorphism type of 
a homotopy $\CP(3)$ is determined by 
the first Pontrjagin class 
(\cite{bro} p.25, \cite{may}, \cite{mayf} Theorem A).
In general dimensions, Masuda and Tsai proved that
for a homotopy complex projective space $M^{2n}$,  
$\delta(M^{2n})$ is divisible by 24 (\cite{mat} Lemma 5.1).  
In 1971, Brumfiel calculated index surgery obstructions 
with target $\CP(4)$ and $\CP(6)$ 
and reported that $\delta(M^{2n})$ is divisible by 16 
for $n=4$ and $n=6$ (\cite{br} Lemma I.5).

Brumfiel's calculation was
based on the calculation 
of the cohomology group of the classifying space $G/O$.
But the details of his calculation are not published. 
So we took one step a way from the surgery theory and looked
closely at Hirzebruch's index theorem.  
We were able to obtain our final result in the following theorem.

\medskip
\noindent {\bf Main Theorem.}\ \ {\sl 
Let $M^{4k}$ be a homotopy $\CP(2k)$.  
Then $\delta(M^{4k})$ is divisible by 16.}
\medskip

In the previous work of the first author \cite{kit}, 
there was a restriction on the 2-order of the integer $k$. 
 In the present paper,  this restriction is completely removed.

\section{Preliminaries:  formal power series 
and elementary number theory}

We shall consider the ring $\Rat[[x]]$ of formal power series 
with rational coefficients.  
An element $f(x)\in \Rat[[x]]$ can be written as
\[ f(x)=\sum_{i\ge0} c_i x^i,\quad (c_i\in\Rat). \]
If $f(x)$ is not $0$, then there exists a non-negative 
integer $j$ such that $c_i=0$ for all $i<j$ and
$c_j\ne 0$.  This number $j$ is the {\it order} of $f(x)$ 
and expressed by $\ord(f(x))$.  If $\rho=\ord(f(x))$, then
$f(x)$ can be expressed as 
\[ f(x)=x^\rho(c_\rho+c_{\rho+1}x+c_{\rho+2}x^2+ \cdots ).\]
From this we see that the quotient field of
$\Rat[[x]]$ is the ring of formal Laurent series
\[ F(x)=\sum_{i\ge r} c_i x^i ,\] 
where $c_i\in \Rat$ and $r\in \Int$,
with only a finite number of negative degree terms.
We shall simply call this expression a formal Laurent series. 
Given a formal Laurent series $F(x)$, 
we shall denote the coefficient of $x^i$ in $F(x)$ by
$(F(x))_i$ or by $(F(x))_{x^i}$.  
The latter notation is usually used 
when we want to specify the variable $x$.  
The coefficient of $x^{-1}$ is called the
formal residue of the formal Laurent series $F(x)$ and
often denoted by $\Res_x(F(x))$.

Let $G(y)=\sum_{i\ge 0} r_i y^i$ be a formal power series
with $r_0=G(0)=0$. 
Then for any formal Laurent series $F(x)$, 
we can perform substitution $x=G(y)$ 
to obtain a new formal Laurent series $F(G(y))$ 
with variable $y$.  This series is expressed
by $F\circ G$.  
When $F(x)$ is a formal power series, 
then $F\circ G$ is
also a formal power series.  
Corresponding to the substitution, we have 
the invariance of formal residues is given 
by the following lemma.

\begin{lemm}\label{twth}
 Let $F(x)=\sum_{i} q_i x^i$ be a formal Laurent
series and $G(y)=\sum_{i} r_i y^i$ 
be a formal power series with $r_0=G(0)=0$
and $r_1\ne 0$.
Then we have
\begin{equation}\label{resdform}
 \Res_x(F(x))=\Res_y (F(G(y))\, G'(y)),  
\end{equation}
where $G'(y)=\sum_i (i+1)r_{i+1}y^i $ 
is the formal derivative of $G(y)$.
\end{lemm}
\begin{proof} Since a formal Laurent series is a linear combination
of $x^n$\ $(n\in\Int)$, it is enough to show the formula for the special case
$F(x)=x^n$.  Unless $n=-1$,  since $F(G(y)) G'(y)=G(y)^{n}G'(Y)$ is a
formal derivative of $G(y)^{n+1}/(n+1)$, its residue with
respect to the variable $y$ is zero.  Thus the formula holds
for $F(x)=x^n$ with $n\ne -1$.  For the case $n=-1$, we have

\begin{align*}
F(G(y))G'(y)& =\frac{G'(y)}{G(y)}
=\frac{\sum_i i r_i y^{i-1}}{\sum_i r_i y^i} \\
&=\frac{\sum_{i\ge 1} i r_i y^{i-1}}{y \sum_{i\ge 1}r_i y^{i-1}}
=\frac{H(y)}{y},
\end{align*}
for some formal power series
 $H(y)=\sum_{i\ge 0} c_i y^i$ with $c_0=1$.
Therefore we have 
\[ \Res_y(G'(y)/G(y))=1.   \]
This completes the proof.
\end{proof}

For a prime $p$,  
$\Int_{(p)}$ is a subring of $\Rat$ 
composed of rational numbers that can be expressed 
as $a/b\ (a,b\in\Int, b\ne 0)$,  with $(b,p)=1$.  
Its invertible element
is of the form $a/b$ with $(a,p)=1$ and $(b,p)=1$. 
If  the coefficients of $F(x)$ and $G(y)$ are 
in $\Int_{(p)}$, and if in addition $G'(0)$ is an invertible element, 
then the substitutions and the formula (\ref{resdform}) in
Lemma \ref{twth} can be performed  
in the same coefficient ring $\Int_{(p)}$.
The following is an inverse function theorem 
in the formal power series theory.

\begin{lemm}\label{invfunc} Let $p$ be a prime and
 $F(x)=\sum_{i\ge0}{q_i x^i}\in\Int_{(p)}[[x]]$ be
a formal power series and 
assume that $F(0)=q_0=0$ and $F'(0)=q_1$ is invertible in
$\Int_{(p)}$.  Then there exists a unique formal power series
$G(y)$ in $\Int_{(p)}[[y]]$ with $G(0)=0$ such that
$F(G(y))=y$.
\end{lemm}
\begin{proof} Let us write $G(y)=\sum_{i\ge 0}r_i y^i$. 
Then we shall show that the coefficients $\{r_i\}$ of $G(y)$ 
are inductively uniquely determined as an element of
$\Int_{(p)}$, by the equality $F(G(y))=y$ starting
from the initial condition $G(0)=r_0=0$.  
Let us write $(G(y))^i=\sum_j r^{(i)}_j y^j$ 
then we can easy see
that $r^{(i)}_j = 0$ for $j<i$. 
From the equality  $F(G(y))=y$,
we have $\sum_i q_i(\sum_j r^{(i)}_j y^j)=y$. 
This reduces to
\begin{equation}\label{inddef}
\sum_i \biggl(\sum_{i\le j} q_i r^{(i)}_j\biggr)y^j = y. \nonumber
\end{equation}
From this we have $q_1 r^{(1)}_1=1$.  
Since $q_1$ is invertible,
$r_1=r^{(1)}_1\in \Int_{(p)}$ is determined.  
For $j\ge 2$, we have
\begin{equation}\label{inddeff} \sum_{i=1}^j q_i r^{(i)}_j 
= q_1 r^{(1)}_j + q_2 r^{(2)}_j +\cdots + q_j r^{(j)}_j=0. 
\end{equation}
If $r_1, r_2, \ldots, r_{j-1}$ are determined 
as elements of $\Int_{(p)}$, then
$r_1^{(i)}, r_2^{(i)}, \ldots, r_j^{(i)}$ are determined 
in $\Int_{(p)}$ for $1<i\le j$ 
as the coefficients of the polynomial
\[ (r_1 y+r_2y^2+\cdots + r_{j-1}y^{j-1})^i. \]
This shows that $r_j^{(2)}, r_j^{(3)},\ldots, r_j^{(j)}$ 
are determined
and from  (\ref{inddeff}), $r_j=r_j^{(1)}$ 
is determined as an
element of $\Int_{(p)}$. 
\end{proof}

Let $\Int_{(p)}[[x]]_1$ stand for the subset of
all formal power series 
with coefficients in the
ring $\Int_{(p)}[[x]]$, with constant term $1$.  
As a corollary to the inverse function theorem 
we have a generalized binomial expansion formula.
\begin{coro}\label{bino}
Let $p$ be a prime and $q$ be a natural number 
relatively prime to $p$.
Then there exists a unique formal power series 
$v(x)\in \Int_{(p)}[[x]]_1$
satisfying $(v(x))^q=1+x$.
\end{coro}

\begin{proof} Consider the formal power series $F(x)=(1+x)^q-1$.
Then we have $F(0)=0$ and $F'(0)=q$ is invertible in $\Int_{(p)}$. 
By Lemma \ref{invfunc}, there exists a unique formal power series
$G(y)\in \Int_{(p)}[[y]]$ satisfying $G(0)=0$ and $(1+G(y))^q-1=y$.  
If we put $v(x)=1+G(x)$, we see that $v(x)^q=1+x$.  
\end{proof}
We shall denote the formal power series $1+G(x)$ in the proof
above by $(1+x)^{1/q}$. 
It is well known that for a rational number $\alpha$, we have a
formal power series expansion
\begin{equation} \label{fbin}
(1+x)^\alpha = 1+\sum_{i\ge 1} \binom{\alpha}{i} x^i, \nonumber
\end{equation}
where  $\binom{\alpha}{i}=\alpha(\alpha-1)\cdots (\alpha-i+1)/i!$.
When $\alpha=1/q$ where $q$ is an integer prime to $p$, then
from the corollary above, the coefficients $\binom{\alpha}{i}$ 
belong to $\Int_{(p)}$.  
This is also true for general $\alpha\in \Int_{(p)}$.

\begin{prop}\label{genfbin} Let $p$ be a prime,
and $f(x)\in\Int_{(p)}[[x]]_1$.
For a natural number $q$
such that $(p,q)=1$, there exists a unique formal power series
$\varphi(x)\in \Int_{(p)}[[x]]_1$ that  satisfies $(\varphi(x))^q=f(x)$.
\end{prop}
\begin{proof}
From Corollary \ref{bino}, there exists $v(x)\in \Int_{(p)}$
satisfying $(v(x))^q=1+x$. Substituting the variable $x$ by
$f(x)-1$ in $v(x)$ we obtain $\varphi(x)=v(f(x)-1)$. Then
we have $(\varphi(x))^q=1+(f(x)-1)=f(x)$.  
The uniqueness of such $\varphi(x)$ can be shown 
by the inductive argument as in
the proof of Lemma \ref{twth}.
\end{proof}

\begin{coro}\label{fbino} Let $p$ be a prime and let $\alpha=m/q$,
$m, q \in \Int$ with $(p,q)=1$.  
Then for any $f(x)\in\Int_{(p)}[[x]]_1$,
then the following formula holds in $\Int_{(p)}[[x]]_1$:
\[ (f(x))^{m/q}=(f(x)^{1/q})^m=(f(x)^m)^{1/q}. \]
\end{coro}

\medskip 
Next we shall introduce the notations 
that will frequently appear in this paper and
explain number theoretic facts 
which will be used in the proofs.  
Let us fix a prime number $p$.  
For any integer $n$ the $p$-order of $n$ is the exponent of
$p$ in the prime factorization of $n$ 
and is denoted by $\nu_p(n)$.  
By convention, we set $\nu_p(0)=\infty$.  
For a rational number $m/n$, where $m$, $n$ $\in\Int$,
we define $\nu_p(m/n)=\nu_p(m)-\nu_p(n)$.  
In the $p$-ary notation of a nonnegative integer
$n=\sum_i n_i p^i$, the sum of digits $\sum_i n_i$ 
is denoted by $\kappa_p(n)$.  

Here we present a fundamental lemma in treating
the $p$-order of the coefficients of powers of a sum.

\begin{lemm}\label{ppw} Let $p$ be a prime and $m$, $n$ 
be non-negative integers.  \par
\noindent{\rm (a)}\quad For variables $x$ and $y$, we have
\[ (x+y)^{p^{m+n}}\equiv (x^{p^m}+y^{p^m})^{p^n} \mod p^{n+1}. \]
{\rm (b)}\quad For variables $x_1, x_2, \ldots, x_r$, we have
\[ (x_1+x_2+\cdots+x_r)^{p^{m+n}}\equiv
(x_1^{p^m}+x_2^{p^m}+\cdots+x_r^{p^m})^{p^n} \mod p^{n+1}. \]
\end{lemm}
\begin{proof}
To prove (a) we use induction on $n$.  It is well known that
the assertion is true for $n=0$.  Assume that (a) is true
for $n$.  Then we can write
\[ (x+y)^{p^{m+n}}=(x^{p^m}+y^{p^m})^{p^n}+p^{n+1}L \]
for some $L\in \Int[x,y]$.  Then taking the $p$-th power, 
we have
\begin{align*} (x+y)^{p^{m+n+1}}
&= (x^{p^m}+y^{p^m})^{p^{n+1}}+\sum_{i=1}^p \binom{p}{i}
(x^{p^m}+y^{p^{m}})^{p^{n}(p-i)}(p^{n+1}L)^i \\
&=(x^{p^m}+y^{p^m})^{p^{n+1}}  \mod p^{n+2}.
\end{align*}
This shows that the assertion is also true for $n+1$.
The proof of (b) goes similarly using the induction on $n$ 
and is omitted.
\end{proof}
From this lemma, we have the following formula.
\begin{coro} \label{fppw}Given a formal power series
$f(x)\in \Int_{(p)}[[x]]$.  If $p$ is a prime, we have
\[ \bigl(f(x)\bigr)^{p^{m+n}
}\equiv \biggl(\sum_{i\ge 0} (c_ix^i)^{p^m}\biggr)^{p^n}
\mod p^{n+1}. \]
\end{coro}

\begin{coro}\label{ppwfps}
Let $p$ be a prime and a formal power series 
$f(x)=\sum_{i\ge0}c_i x^i$,  let us express its
$l$-th power as
\[ \bigl(f(x)\bigr)^l = \sum_{i\ge 0} c_i^{(l)} x^i . \]
If all the coefficents of $f(x)$ are in $\Int_{(p)}$, then we have 
\[ \nu_p(c_i^{(l)})\ge \nu_p(l)-\nu_p(i).\]
\end{coro}
\begin{proof} Let us fix $i$.  
The assertion trivially holds
if $\nu_p(l)\le\nu_p(i)$. 
So we may assume that
$\nu_p(l)\ge \nu_p(i)+1$. 
We set $m=\nu_p(i)+1$ and $n=\nu_p(l)-m$.  Then since
$m+n=\nu_p(l)$, we can write $l=p^{m+n}q$ for some $q\in\Int_+$ 
with $(p,q)=1$. 
From Corollary \ref{fppw}, we have
\[ \bigl(f(x)\bigr)^l=\bigl(f(x)\bigr)^{p^{m+n}q}
\equiv \biggl(\sum_{j\ge0}\bigl(c_jx^j\bigr)^{p^m}\biggr)^{p^nq}
\mod p^{n+1}. \]
This shows that since $i$ is not divisible by $p^m$, 
$c_i^{(l)} \equiv 0 \mod p^{n+1}$ holds.  That is
\[ \nu_p(c_i^{(l)})\ge n+1 = \nu_p(l)-m+1=\nu_p(l)-\nu_p(i).  \]
\end{proof}

\medskip

In this paper
we are interested in the case where $p=2$ and
we shall only consider the case $p=2$ from now on.

\begin{lemm}\label{twon}
Let $n, k$ be integers with $0\le k\le n$.
Then we have the following.\par
{\rm (a)}\qquad $\nu_2(n!)=n-\kappa_2(n)$.\par
{\rm (b)}\qquad $\nu_2(\binom{n}{k})=
\kappa_2(k)+\kappa_2(n-k)-\kappa_2(n)$.\par
{\rm (c)}\qquad Let $n=\sum_i n_i 2^i$,
\ $k=\sum_i k_i 2^i$ be the
binary notations of $n$ and $k$. 
Then the binomial coefficient $\binom{n}{k}$ is even
if and only if there exists $i$ such that $n_i<k_i$.
\end{lemm}
\begin{proof}
It is not difficult to see that the two 
sequences $\{q_n^{(i)}\}, \ (i=1,2)$ defined by
$q_n^{(1)}=\nu_2(n!)$ and by $q_n^{(2)}=n-\kappa_2(n)$ 
both satisfy the
same inductive formula
\[ q_0^{(i)}=0 , \qquad q_n^{(i)}=[n/2]+q_{[n/2]}^{(i)}, \]
where $[t]$ denotes the largest integer not exceeding $t$.  
This formula uniquely determines the sequences $\{q_n^{(i)}\}$ 
and this fact proves (a).
(b) follows immediately from (a).
To show (c), if there exists a column $i$ 
such that $n_i<k_i$ then
in the addition process of $k$ and $n-k$ in binary form,
there exists a column where the digit addition carries 1 
to the next column.  If such a column exists, 
 there arises a decrease of sum of digits as in
\[  \kappa_2(k)+\kappa_2(n-k)>\kappa_2(n) \]
and this proves our assertion.
\end{proof}

\begin{lemm}\label{twtw}
Let $i$,  $n$ and $m$ be natural numbers 
and assume that $n$ and $m$ are odd.  
Then we have the following.\par
{\rm (a)}\quad $ \nu_2(n^m+1)=\nu_2(n+1)\quad \mbox{and}\quad 
\nu_2(n^m-1)=\nu_2(n-1). $\par
{\rm (b)}\quad $\nu_2(n^{2i}-1)=\nu_2(n^2-1)+\nu_2(i).$ \par
{\rm (c)}\quad $\nu_2(n^{i}-(-1)^i)=
\begin{cases}
 \nu_2(n+1), & \mbox{if}\ i\ \mbox{is\ odd} \\
\nu_2(n^2-1)+\nu_2(i)-1, & \mbox{if}\  i\  \mbox{is even.}
\end{cases}
$
\end{lemm}
\begin{proof} (a) follows immediately from the factorizations
\begin{align*}
& n^m+1=(n+1)(n^{m-1}-n^{m-2}+\cdots -n+1)\\
& n^m-1 = (n-1)(n^{m-1}+n^{m-2}+\cdots +n+1). 
\end{align*}
To show (b), in view of (a) we may assume, 
without loss of generality, that
$i=2^e$.  Then from the factorization
\[ n^{2i}-1=(n^2-1)(n^2+1)(n^{2^2}+1)\cdots (n^{2^e}+1), \]
and from the fact that $\nu_2(n^{2^r}+1)=1$ for $r\ge1$, 
we have the conclusion by (a) and (b). 
\end{proof}

\begin{lemm}\label{cnton}
 Let $i_1, i_2, \ldots, i_s$ be non-negative
integers. \par
\noindent {\rm (a)} \ If $i_1\ge 1$, then
\[ \nu_2(i_1)+\kappa_2(i_1)+\kappa_2(i_2) \ge \nu_2(i_1+i_2)+1.\]
\noindent {\rm (b)}\ If $i_1\ge 1$, then
\[ \nu_2(i_1)+\kappa_2(i_1)+\kappa_2(i_2)+\cdots
+\kappa_2(i_s)\ge \nu_2(i_1+\cdots+i_s)+1. \]
\end{lemm}
\begin{proof} The proof of (a) is divided into several
cases.\par
\noindent Case : $i_1+i_2$ odd.  Then since $\kappa_2(i_1)\ge 1$
and $\nu_2(i_1+i_2)=0$, we get the assertion.\par
\noindent Case: both $i_1$ and $i_2$ odd.  Express $i_1$ and $i_2$ 
in the binary notations:
\[ i_1=\sum_{j\ge 0}s_j2^j,\ \ i_2=\sum_{j\ge 0} t_j2^j , \]
where $s_j, t_j = 0,1$.
Let $d=\nu_2(i_1+i_2)$ then we have $s_0=t_0=1$, and for
each $j$ with $1\le j \le d-1$, $s_j+t_j=1$.  From this
we have
\[ \nu_2(i_1)+\kappa_2(i_1)+\kappa_2(i_2)\ge d+1=\nu_2(i_1+i_2)+1. \]
\noindent Case : both $i_1$ and $i_2$ even.
Let $t=\min(\nu_2(i_1), \nu_2(i_2))$. Then we can write
$i_1=2^t j_1$, $i_2=2^t j_2$ and 
$j_1$ or $j_2$ is odd.  Therefore we have
\[ \nu_2(j_1)+\kappa_2(j_1)+\kappa_2(j_2)\ge \nu_2(j_1+j_2)+1 . \]
Adding $t$ to both sides, we have
\[ \nu_2(i_1)+\kappa_2(i_1)+\kappa_2(i_2)\ge \nu_2(i_1+i_2)+1 .\]
This completes the proof of (a).\par
(b) follows immediately using (a) as follows:
\[ \nu_2(i_1)+ \kappa_2(i_1)+\cdots+\kappa_2(i_s)
\ge \nu_2(i_1)+\kappa_2(i_1)+\kappa_2(i_2+\cdots+i_s)
\ge \nu_2(i_1+\cdots+i_s)+1.\]
\end{proof}

Now we shall present definitions, notations and basic facts
about Bernoulli numbers and Hirzebruch power series.
Recall that the Bernoulli numbers $\beta_n$ 
are defined by 
\begin{equation*}
\frac{x}{e^x-1}=\sum_{n\ge 0} \frac{\beta_n}{n!}x^n.
\end{equation*}
We know that $\beta_0=1$, $\beta_1=-1/2$ and 
$\beta_{2i+1}=0$ if $i\ge 1$.  For $i\ge 1$, we set
$B_i=(-1)^{i+1}\beta_{2i+1}$.  These numbers $B_i$ are
are positive and also called Bernoulli numbers. 
By an easy calculation we obtain Hirzebruch power series
\[ h(x)=\frac{x}{\tanh x}= 1+\sum_{i\ge 1}
\frac{(-1)^{i+1}2^{2i}B_i}{(2i)!}x^{2i}.  \]
To simplify our notation, we put
$ a_i=(h(x))_{2i}=(-1)^{i+1}2^{2i}B_i/(2i)!$.  Then
$h(x)=\sum_{i\ge 0} a_i x^{2i}$. 
The 2-orders of these coefficients $a_i$ are given
by the following lemma which immediately implies
that $h(x)$ is in $\Int_{(2)}[[x]]$.

\begin{lemm}\label{nyua} $\nu_2(a_i)=\kappa_2(i)-1$  for all $i$ .
\end{lemm}
\begin{proof} You can prove the assertion by using 
the theorem of Clausen-von Staudt  $\nu_2(B_i)=-1$,
 (see \cite{haw}).  However to keep our exposition self-contained,
we shall present here an alternative elementary proof.
By multiplying $h(x)$ by $\sinh 2x$, we have
\[ h(x)\sinh 2x = x(\cosh 2x +1). \]
When we take the $(2n+1)$-th derivative ($n\ge1$) using the
general Leibniz rule, we get
\[ \sum_{i=0}^{2n+1}\binom{2n+1}{i}h^{(i)}(x) (\sinh 2x)^{(2n+1-i)}
= x 2^{2n+1}\sinh 2x +(2n+1) 2^{2n}\cosh 2x. \]
Substituting $x=0$ we have
\[ \sum_{j=0}^n \binom{2n+1}{2j}h^{(2j)}(0) 2^{2n+1-2j}
= (2n+1)2^{2n}. \]
Division by $2^{2n}$ gives
\[ \sum_{j=0}^n \binom{2n+1}{2j}\frac{h^{(2j)}(0)}{2^{2j-1}}
= 2n+1.
\]
We put $u_j=h^{(2j)}(0)/(2^{2j-1})$ and since $h(0)=1$ we have
\begin{equation}\label{hizind}
 \sum_{j=1}^n \binom{2n+1}{2j}u_j = 2n-1. 
\end{equation}
When $n=1$, we have $u_1=1/3$  and $u_1\equiv 3 \mod 4$ in
$\Int_{(2)}$.  We will show that $u_j\in\Int_{(2)}$ and
$u_j \equiv 1 \mod 4$ for $j\ge 2$ by induction on $j$.
Suppose that $u_j\in \Int_{(2)}$ and $u_j\equiv 1\mod4$
for $2\le j\le n-1$ \ $n\ge 3$.
We know from the binomial expansion formula that
\[ \sum_{j=1}^{n}\binom{2n+1}{2j}=2^{2n}-1. \]
Subtracting this from (\ref{hizind}), we have
\[ \sum_{j=1}^n\binom{2n+1}{2j}(u_j-1)=2n-2^{2n}\equiv 2n \mod 4. \]
On the other hand, from the inductive assumption we have
\begin{align*}
\sum_{j=1}^n \binom{2n+1}{2j}(u_j-1)&\equiv
2\binom{2n+1}{2}+(2n+1)(u_n-1) \mod 4  \\
& = 2n(2n+1)+(2n+1)(u_n-1).
\end{align*}
Thus we have 
\[ 2n(2n+1)+(2n+1)(u_n-1)\equiv 2n \mod 4 .\]
Hence $(2n+1)(u_n-1)\equiv 0 \mod 4$.  Therefore
we have $u_n\equiv 1 \mod 4$.  From this we
have $\nu_2(h^{(2j)}(0))=2j-1$.   Finally we have
\[ \nu_2(a_i)=\nu_2(h^{(2i)}(0))-\nu_2((2i)!)
= (2i-1)-(2i-\kappa_2(2i))=\kappa_2(i)-1. \]
\end{proof}

We shall define another formal power series
\[ g(x)=\frac18\left(\frac{h(3x)}{h(x)}-1\right)
=\frac18\left(\frac{3\tanh x}{\tanh 3x}-1\right)
=\frac{\tanh^2 x}{3+\tanh^2 x} . \]
We shall simply express this formal power series as
\[
g(x)=\sum_{i\ge1}b_i x^{2i} . \]

Since all the coefficients of the formal power series
of
\[ \tanh x =\sum_{i\ge 1}
\frac{(-1)^{i+1}2^{2i}(2^{2i}-1)B_i}{(2i)!}x^{2i-1} \]
belong to $\Int_{(2)}$, the coefficients $b_i$ of $g(x)$ 
all belong to $\Int_{(2)}$.  About their $2$-orders we have
the following lemma.  Remark that this fact does not
follow immediately from the theorem of Clausen-von Staudt.
\begin{lemm}\label{nyub} $\nu_2(b_i)=\kappa_2(i)-1$ for all $i\ge 1$.
\end{lemm}
\begin{proof}
The proof is done using a similar argument of the 
proof of the previous lemma.  

We have
\[ g(x)=\frac{\tanh^2 x}{3+\tanh^2 x}
=\frac{\sinh^2 x}{3\cosh^2 x+\sinh^2 x}
=\frac{\cosh 2x-1}{4\cosh 2x + 2}, \]
and \[ (2\cosh 2x+1)g(x)=(\cosh 2x-1)/2 . \]
Taking the $2n$-th derivative ($n\ge 1$) of both hands, we have
\[
g^{(2n)}(x)\bigl(2\cosh 2x+1\bigr)+
\sum_{i=1}^{2n}\binom{2n}{i}g^{(2n-i)}(x)
\bigl(2\cosh 2x)^{(i)}= 2^{2n-1}\cosh 2x .\]
Substituting $x=0$ in this expression, we have
\[ 3 g^{(2n)}(0)+\sum_{j=1}^n \binom{2n}{2j}2^{2j+1}g^{(2n-2j)}(0)=2^{2n-1}. \] 
Define $u_j=g^{(2j)}(0)/2^{2j-1}$, then we have
\begin{equation}\label{unb}
3u_n+2 \sum_{j=1}^{n-1} \binom{2n}{2j}u_{n-j}=1 .
\end{equation}
By putting $n=1$ to (\ref{unb}), we have
$u_j=1/3\equiv 3 \mod 4.$  We shall prove that for all 
$u_j\equiv 3\mod 4$ for all $j$.  We use induction and
let us suppose that $u_j\equiv 3 \mod 4$ for $j< n$.
Then from (\ref{unb}), we have
\[ 3u_n+6\sum_{j=1}^{n-1}\binom{2n}{2j}\equiv 1 \mod 4 .\]
Since $\sum_{j=1}^{n-1}\binom{2n}{2j}=2^{2n-1}-2$ is even,
we have $3u_n\equiv 1 \mod 4$ and this implies
that $u_n\equiv 3 \mod 4$. Thus for all $j$,
 $u_j\equiv 3 \mod 4$.  Therefore
\[
\nu_2(b_i)=\nu_2(g^{(2i)}(0)/(2i)!)
=(2i-1)-(2i-\kappa_2(2i))=\kappa_2(i)-1. \]
\end{proof}

\section{The index theorem for a homotopy $\CP(2k)$}

Let $\eta$ be the canonical complex line bundle over
$\CP(2k)$ whose first Chern class $x=c_1(\eta)$ is
generates the cohomology ring 
$H^*(\CP(2k);\Int)=\Int[x]/(x^{2k+1})$.
If  $f : M^{4k} \rightarrow \CP(2k)$ is a homotopy $\CP(2k)$,
then there exists 
a fiber homotopically trivial vector bundle $\zeta$
over $\CP(2k)$ such that the tangent bundle $\tau(M)$
is stably isomorphic to the pullback of 
$\tau(\CP(2k))\oplus \zeta $ by $f$ :
\[ \tau(M)\overset{s}{\sim} f^*(\tau(\CP(2k))\oplus \zeta). \]
Using Hirzebruch's index thorem, we see that
\[ \mbox{Index}(M)=\langle \mathcal{L}(M),[M] \rangle 
=\langle \mathcal{L} (\zeta)h(x)^{2k+1}, [\CP(2k)]\rangle, \]
where $\mathcal{L}$ denotes Hirzebruch's $L$-class $\sum_j L_j$ 
associated to the power series $h(x)$.
Since $\delta(M)x^2$ coincides with $p_1(\zeta)$, we have
to examine the Pontrjagin class of $\zeta$ 
when $\mbox{Index}(M)=1$ holds.

Let $\omega\in \widetilde{KO}(\CP(2k))$ denote the realification of
$\eta-1_\cmplx\in \widetilde{K}(\CP(2k))$.  It is known that 
$\widetilde{KO}(\CP(2k))$ is a free abelian group
generated by 
$\omega^j$\ $(j=1,2,\ldots,k)$ (\cite{san}).  
The real Adams operation on $\omega$ is given by the formula
\[ \psi_\Real^j(\omega)=T_j(\omega), \]
where $T_j(z)$ is a polynomial of $z$ having degree $j$ characterized by
the property
\[ T_j(t+t^{-1}-2)=t^j+t^{-j}-2 . \]
Since the coefficient of $z^j$ in $T_j(z)$ is one, we may
take $\psi_\Real^j(\omega)=T_j(\omega)$ \ $(1\le j \le k)$ 
as generators of $\widetilde{KO}(\CP(2k))$.

According to the solution of the
Adams-conjecture, the kernel of the $J$-map 
coincides with $\mbox{Image}(\psi^3_\Rat-1)$
when localized at $2$.  Therefore 
when we put $\zeta_j=(\psi_\Real^3-1)\psi_\Real^j(\omega)$, 
the fiber homotopically trivial vector bundle $\zeta$ can
be written as
\[ q\zeta=n_1\zeta_1+n_2\zeta_2+\cdots+n_k\zeta_k ,\]
for some integers $n_1, n_2, \ldots, n_k$ and an odd integer $q$.
Therefore we may write
\[ \zeta=m_1\zeta_1+m_2\zeta_2+\cdots+m_k\zeta_k ,\]
where $m_j, m_2, \ldots, m_k$ belong to $\Int_{(2)}$.  
We first calculate the total Pontrjagin class of $\psi_\Real^j(\omega)$.
We first note that
\begin{align*} \psi_\Real^j(\omega)\otimes \cmplx 
&= \psi_\cmplx^j(\omega\otimes \cmplx)
=\psi_\cmplx^j(\eta+\bar{\eta}-2_\cmplx) \\
&= \psi_\cmplx^j(\eta)+\psi_\cmplx^j(\bar{\eta})-2_\cmplx
=\eta^j+\bar{\eta}^j-2_\cmplx,
\end{align*}
whose Chern class is equal to $(1+jx)(1-jx)=1-j^2x^2$.  Therefore
the total Pontrjagin classes are given by
\[ p(\psi_\Real^j(\omega))=1+j^2x^2 ,\]
\[ p(\zeta_j)=p(\psi_\Real^{3j}(\omega)-\psi_\Real^j(\omega))
=\frac{1+(3j)^2x^2}{1+j^2x^2}, \]
and
\[ p(\zeta)=\prod_{j=1}^k p(\psi_\Real^j(\omega))^{m_j}
=\prod_{j=1}^k \left(\frac{1+(3j)^2x^2}{1+j^2x^2}\right)^{m_j} . \]
From this we have
\[ p_1(\zeta)= 8 \sum_{j=1}^k j^2 m_j. \]
The $\mathcal{L}$ class of $\zeta$ is written as
\[ \mathcal{L}(\zeta)=\prod_{j=1}^k \mathcal{L}(\zeta_j)^{m_j}
= \prod_{j=1}^k \biggl(\frac{h(3jx)}{h(jx)}\biggr)^{m_j}  
=\prod_{j=1}^k (1+8g(jx))^{m_j}.   \]

We calculate the index of $M^{2k}$:
\begin{align*}
&\mbox{Index}(M) = \langle \mathcal{L}(\zeta) h(x)^{2k+1}, [\CP(2k)]\rangle \\
&= \left(\mathcal{L}(\zeta)h(x)^{2k+1}\right)_{2k}
=\left(\prod_{j=1}^k (1+8g(jx))^{m_j}h(x)^{2k+1}\right)_{2k} \\
&= 1+ 8 \sum_{j=1}^k m_j (g(jx)h(x)^{2k+1})_{2k} \\
  &\qquad + \sum_{s\ge 2}8^s \sum_{i_1+\cdots +i_k=s}
\binom{m_1}{i_1}\cdots \binom{m_k}{i_k}
\biggl( g(x)^{i_1}g(2x)^{i_2}\cdots 
g(kx)^{i_k} h(x)^{2k+1}\biggr)_{2k},
\end{align*}
where $\bigl( f(x) \bigr)_{j}$ denotes the coefficient of $x^j$
in the formal power series $f(x)$.
We shall use the following notations:
\[ C(j_1, j_2, \cdots, j_s)=(g(j_1x)g(j_2x)\cdots
 g(j_sx)h(x)^{2k+1})_{2k}, \]
\[ D(i_1, i_2,\ldots, i_k)
= C(\underbrace{1,\ldots,1}_{i_1}, \underbrace{2,\ldots,2}_{i_2}, 
\ldots,\underbrace{k,\ldots,k}_{i_k}). \]
Then we have
\[ \mbox{Index}(M)=1+8\sum_{j=1}^k m_j C(j)
+ \sum_{s\ge 2}8^s \sum_{i_1+\cdots+i_k=s}
\binom{m_1}{i_1}\cdots\binom{m_k}{i_k}D(i_1,\ldots,i_k). \]
Since $\mbox{Index}(M)=1$, we have 
\begin{equation}\label{idx}  \sum_{j=1}^k m_j C(j)+\sum_{s\ge 2}8^{s-1}
\sum_{i_1+\cdots+i_k=s}
\binom{m_1}{i_1}\cdots\binom{m_k}{i_k}D(i_1,\ldots,i_k)=0. 
\end{equation}
Our target is to show that
$p_1(\zeta)$ is divisible by 16 from the condition (\ref{idx}).  This is equivalent to
the claim that
 $ \sum_{j=1}^k j^2 m_j $ is even.  This is also equivalent
to $\sum_{j:odd}m_j$ is even.  

\begin{lemm}\label{cs}
\begin{eqnarray}\label{csprop}
\bigl(g(x)^s h(x)^{2k+1}\bigr)_{2k}&=&
C(\underbrace{1,\ldots,1}_{s}) = D(s,0,\ldots,0)=\left(\frac{1}{(3+x)^s(1-x)}\right)_{k-s} \nonumber\\
&=&\frac{1}{4^s}\biggl(\frac{1}{1-x}+\frac{1}{3+x}+\frac{4}{(3+x)^2}
+\cdots+\frac{4^{s-1}}{(3+x)^s} \biggr)_{k-s} \nonumber \\
&=&\frac{1}{4^s 3^k}\biggl(3^k+(-1)^{k-s}\sum_{i=0}^{s-1} \binom{k-s+i}{i}
3^{s-1-i}4^i\biggr). \nonumber
\end{eqnarray}
\end{lemm}
\begin{proof}
\begin{align*}
\bigl(g(x)^s &h(x)^{2k+1}\bigr)_{2k}
=\Biggl( \biggl(\frac{\tanh^2x}{3+\tanh^2x}\biggr)^s 
\biggl(\frac{x}{\tanh x}\biggr)^{2k+1}\Biggr)_{x^{2k}} \\
&=\Res_x\Biggl(\biggl(\frac{\tanh^2x}{3+\tanh^2x}\biggr)^s 
\frac{1}{\tanh^{2k+1} x}\Biggr) \\
\intertext{by putting $y=\tanh x$,} 
& =\Res_y\Biggl(\biggl(\frac{y^2}{3+y^2}\biggr)^2\frac{1}{y^{2k+1}(1-y^2)}\Biggr) \\
&=\Res_y\Biggl(\frac{1}{y^{2k+1-2s}(3+y^2)^s(1-y^2)} \Biggr) \\
&= \Bigg( \frac{1}{(3+y^2)^s(1-y^2)}\Biggr)_{y^{2k-2s}} 
=\Biggl(\frac{1}{(3+x)^s(1-x)}\Biggr)_{x^{k-s}}.
\end{align*}
By induction we can show that
\[ \frac{1}{(3+x)^s(1-x)}=\frac{1}{4^s}
\Biggl(\frac{1}{1-x}+\frac{1}{3+x}+\frac{4}{(3+x)^2}
+\cdots+ \frac{4^{s-1}}{(3+x)^s} \Biggr). \]
From this we have
\begin{align*}
&\Biggl(\frac{1}{(3+x)^s(1-x)} \Biggr)_{k-s} \\
&= \frac{1}{4^s}\Biggl(1 +\biggl(-\frac{1}{3}\biggr)^{k-1}
\biggl(\frac{1}{3}+\frac{4}{3^2}\binom{k-s+1}{1}
+\frac{4^2}{3^3}\binom{k-s+2}{2}
+\cdots+\frac{4^{s-1}}{3^s}\binom{k-1}{s-1} \biggr)\Biggr) \\
&=\frac{1}{4^s3^k}\Biggl( 3^k+(-1)^{k-s}
\sum_{i=0}^{s-1}\binom{k-s+i}{i}3^{s-1-i}4^i \Biggr).
\end{align*}
\end{proof}

To simplify our notation, from now on we shall denote $\nu_2(k)$ by $r$.  
As a special case $s=1$, we have the following corollary.

\begin{coro}\label{cone}  {\rm (a)\ }\ 
 $C(1)=(3^k-(-1)^k)/(4\cdot 3^k), $\ and \
{\rm (b)\ }$\nu_2(C(1))=r$. \end{coro}
\begin{proof}
(a) follows from the previous lemma.  
(b) is a result of (a) using Lemma \ref{twtw}\ (c).
\end{proof}

\begin{prop}\label{propa}
If $j$ is odd, then
$\nu_2(C(j))=r$. 
\end{prop}
\begin{proof}
The case for $j=1$ is given in Corollary \ref{cone}.  
For any other odd number $j$, we have
\begin{align*}
C(j)& -C(1)= \biggl(\bigl( g(jx)-g(x) \bigr)h(x)^{2k+1}\biggr)_{2k} 
= \biggl( \bigl( g(jx)-g(x) \bigr) h(x) h(x)^{2k}\biggr)_{2k} \\
= &\biggl( \sum_{i_1\ge 1}b_{i_1}(j^{2i_1}-1) x^{2i_1}
\sum_{i_2\ge 0} a_{i_2}x^{2i_2}\sum_{i_3\ge 0} a^{(2k)}_{i_3} x^{2i_3}\biggr)_{2k} 
=\sum_{i_1+i_2+i_3=k, i_1\ge 1}(j^{2i_1}-1)b_{i_1}a_{i_2}a^{(2k)}_{i_3}.
\end{align*}
Here we put $U=(j^{2i_1}-1)b_{i_1}a_{i_2}a^{(2k)}_{i_3}$.  
From Lemma \ref{twtw} (b) we have
\[ \nu_2(j^{2i_1}-1)b_{i_1}=\nu_2(j^2-1)+\nu_2(i_1)+\kappa_2(i_1)-1
\ge \nu_2(i_1)+\kappa_2(i_1)+2. \]
From Lemma \ref{nyub}, we have
\[ \nu_2(a_{i_2})=\kappa_2(i_2)-1. \]
From Lemma \ref{ppwfps}, we have
\[
\nu_2(a^{(2k)}_{i_3})\ge \nu_2(2k)-\nu_2(i_3)=r+1-\nu_2(i_3). \]
Thus we have $\nu_2(U)\ge \nu_2(i_1)+\kappa_2(i_1)+\kappa_2(i_2)+1
+\nu_2(a^{(2k)}_{i_3})$.  
From this, we have
\[ \nu_2(U)\ge \nu_2(i_1)+\kappa_2(i_1)+\kappa_2(i_2)+1+\nu_2(a^{(2k)}_{i_3})
\ge \nu_2(k-i_3)+2+\max(r+1-\nu_2(i_3),0). \]
If $i_3=0$,  from Lemma \ref{cnton}, we have
$\nu_2(U)\ge \nu_2(i_1+i_2)+2=r+2$. 
If $i_3\ge 1$ and $\nu_2(i_3)\le r$, then since $\nu_2(k-i_3)\ge \nu_2(i_3)$,
we have $\nu_2(U)\ge r+3$.  If $\nu_2(i_3)>r$ , then
since $\nu_2(k-i_3)=r$, we have
$\nu_2(U)\ge r+2$.  Therefore we always have $\nu_2(U)\ge r+2$.  This
shows that $\nu_2(C(j)-C(1))\ge r+2$.  Therefore we have
$\nu_2(C(j))=r$. 
\end{proof}

\begin{prop}\label{propb}
Let $j$ be an even natural number.  Then $\nu_2(C(j))\ge r+1$.
\end{prop}
\begin{proof}
As in the proof of Proposition \ref{propa}, we have
\[ C(j)=\sum_{i_1+i_2+i_3=k, i_1\ge 1}j^{2i_1}b_{i_1}a_{i_2}a^{(2k)}_{i_3}. \]
Put $B=j^{2i_1}b_{i_1}a_{i_2}a^{(2k)}_{i_3}$.   
Since $\nu_2(j^{2i_1})\ge 2i_1$, $\nu_2(b_{i_1})=\kappa_2(i_1)-1$, 
$\nu_2(a_{i_2})=\kappa_2(i_2)-1$ and $\nu_2(a^{(2k)}_{i_3})\ge
\max(r+1-\nu_2(i_3),0)$, we have
\begin{align*} \nu_2(B) 
&\ge 2i_1+\kappa_2(i_1)+\kappa_2(i_2)-2+\max(r+1-\nu_2(i_3),0) \\
&\ge 2i_1+\nu_2(k-i_3)-\nu_2(i_1)-1+\max(r+1-\nu_2(i_3),0).
\end{align*}
If $i_3=0$, then since $i_1\ge\nu_2(i_1)+1$, we have
\[ \nu_2(B)\ge 2i_1+r-\nu_2(i_1)-1 \ge r+i_1\ge r+1.  \]
If $i_3\ge 1$ and $\nu_2(i_3)\le r$, then we have
$\nu_2(k-i_3)\ge \nu_3(i_3)$.  Hence we have
\[ \nu_2(B)\ge 2i_1+\nu_2(i_3)-\nu_2(i_1)-1+(r+1-\nu_2(i_3))
= r+2i_1-\nu_2(i_1)\ge r+1+(i_1-\nu_2(i_1))\ge r+2.\]
Finally if $i_3\ge 1$ and $\nu_2(i_3)> r$, then we have
\[ \nu_2(B)\ge 2i_1+r-\nu_2(i_1)-1\ge i_1+r\ge r+1. \]
Thus in all cases, we have $\nu_2(B)\ge r+1$.  This proves
our assertion.
\end{proof}

\begin{lemm}\label{csord}
$\nu_2(C(\underbrace{1,\ldots,1}_{s}))\ge r+2-2s.$
\end{lemm}
\begin{proof} The assertion is true for $s=1$ by Corollary \ref{cone}.
So we assume that $s\ge 2$. 
From Lemma \ref{cs} we have
\begin{align*}
C(\underbrace{1,\cdots,1}_s)&=\frac{1}{4^s3^k}\biggl(3^k+(-1)^{k-s}
\sum_{i=0}^{s-1}\binom{k-s+i}{i}3^{s-1-i}4^i\biggr) \\
&=\frac{1}{4^s3^k}\biggl(3^k+(-1)^{k-s}\sum_{i=0}^{s-1}
\frac{(k-s+1)(k-s+2)\cdots(k-s+i)}{i!}3^{s-1-i}4^i\biggr) \\
&=\frac{1}{4^s3^k}\biggl( 3^k+(-1)^{k-s}\sum_{i=0}^{s-1}
\frac{3^{s-1-i}4^i}{i!}(k-s+1)(k-s+2)\cdots(k-s+i)\biggr).
\end{align*}
Define a polynomial with variable $k$ :
\[ w(k)=\sum_{i=1}^{s-1}\frac{3^{s-1-i}4^i}{i!}(k-s+1)(k-s+2)\cdots(k-s+i) . \]
We have
\[ \nu_2\biggl(\frac{3^{s-1-i}4^i}{i!}\biggr)
= 2i-(i-\kappa_2(i))=i+\kappa_2(i)\ge 2. \]
Therefore $w(k)$ is a polynomial in $k$ with coefficients in $4\Int_{(2)}$.
As for the constant term $w(0)$, we have
\begin{align*}
w(0) &= \sum_{i=1}^{s-1}\frac{3^{s-1-i}4^i}{i!}(-s+1)(-s+2)\cdots(-s+i) \\
&=\sum_{i=1}^{s-1}\frac{(s-1)(s-2)\cdots(s-i)}{i!}3^{s-1-i}(-4)^i \\
&=\sum_{i=1}^{s-1}\binom{s-1}{i}3^{s-1-i}(-4)^i
= (3+(-4))^{s-1}-3^{s-1}=(-1)^{s-1}-3^{s-1}.
\end{align*}
Therefore we have
\begin{align*}
C(\underbrace{1,\cdots,1}_s)&=\frac{1}{4^s3^k}\bigl(3^k+(-1)^{k-s}(3^{s-1}+w(k))\bigr) \\
&= \frac{1}{4^s3^k}\bigl( 3^k +(-1)^{k-s}(3^{s-1}+w(0))+(-1)^{k-s}(w(k)-w(0))\bigr)
\\ 
&=\frac{1}{4^s3^k}\bigl(3^k-(-1)^k+(-1)^{k-s}(w(k)-w(0))\bigr).
\end{align*}
By Lemma \ref{twtw}, $\nu_2(3^k-(-1)^k)=r+2$.  
As a polynomial in  $k$,
all the coefficients of $w(k)-w(0)$ in have 2-orders at least 2.
Therefore $\nu_2(w(k)-w(0))\ge r+2$.  
This proves that $\nu_2(C(1,\ldots,1))\ge r+2-2s$.
\end{proof}

\begin{lemm}\label{diffcsodd}
Let $j_1, j_2,\ldots,j_s $ and $j'_1$ be odd natural numbers.
Then $\nu_2(C(j_1,j_2,\ldots,j_s)-C(j'_1, j_2,\ldots,j_s))\ge
r+3-s.$
\end{lemm}
\begin{proof} We have
\begin{align*}
 C(j_1,j_2,\ldots&,j_s)-C(j'_1,j_2,\ldots,j_s)  \\
& = \sum_{i_1+\cdots+i_s+i_{s+1}+l=k}
(j^{2i_1}_1-j'^{2i_1}_1)j_2^{2i_2} \cdots j_s^{2i_s}
b_{i_1}b_{i_2}\cdots b_{i_s}a_{i_{s+1}}a^{(2k)}_l.
\end{align*}
We put $V=(j_1^{2i_1}-j'^{2i_1}_1)j_2^{2i_2}\cdots j_s^{2i_s}
b_{i_1}b_{i_2}\cdots b_{i_s}a_{i_{s+1}}a^{(2k)}_l$.  
From Lemma \ref{twtw}, we have
\[ \nu_2(j^{2i_1}-j'^{2i_1}_1)\ge 3+\nu_2(i_1). \]
From Lemmas \ref{cnton}, \ref{nyua} and \ref{nyub}, we have
\[ \nu_2(b_{i_1}\cdots b_{i_s}a_{i_{s+1}})
= \kappa_2(i_1)+\cdots+\kappa_2(i_{s+1})-(s+1)
\ge \nu_2(k-l)-\nu_2(i_1)-s. \] 
Hence we have
\[ \nu_2(V)\ge \nu_2(k-l) + 3 - s + \nu_2(a^{(2k)}_l). \]
If $l=0$, then we have 
\[ \nu_2(V)\ge r+3-s .\]
If $l\ge1$ and $\nu_2(l)\le r$, then from Corollary \ref{ppwfps}
we have
\[ \nu_2(V)\ge \nu_2(l) + 3 - s +(r+1-\nu_2(l))
= r+4-s . \]
If $l\ge1$ and $\nu_2(l)>r$, then we have
\[ \nu_2(V)\ge r+3-s. \]  
Therefore $\nu_2(V)\ge r+3-s$ holds.  
\end{proof}

\begin{coro}\label{csodd}
If $j_1, \ldots, j_s$ are all odd,
then we have
$\nu_2(C(j_1,\ldots,j_s)) \ge r+2-2s$.
\end{coro}
\begin{proof}
From Lemma \ref{diffcsodd}, we have
$\nu_2(C(j_1,\ldots,j_s)-C(1,\ldots,1))\ge r+3-s$.
On the other we know from Corollary \ref{csord}
that $C(1,\ldots,1)$ satisfies
$\nu_2(C(1,\ldots,1))\ge r+2-2s$.
Since $r+3-s > r+2-2s$ holds, we have 
$\nu_2(C(j_1,\ldots,j_s))\ge r+2-2s$. 
\end{proof}

\begin{lemm}\label{cseven}
If there exists at least one even number
in $j_1, \ldots, j_s$, then
we have 
\[ \nu_2(C(j_1,\ldots,j_2))\ge r+2-s .\] 
\end{lemm}
\begin{proof}
Without loss of generality, we may assume that
$j_1$ is even.  
We have
\[ C(j_1,\ldots,j_s)=\sum_{i_1+\cdots+i_{s+1}+l=k} j_1^{2i_1}\cdots
j_s^{2i_s}b_{i_1}\cdots b_{i_s}a_{i_{s+1}}a^{(2k)}_l .\]
We put $ U=j_1^{2i_1}\cdots j_s^{2i_s}b_{i_1}\cdots b_{i_s}
a_{i_{s+1}}a^{(2k)}_l$.  
Since $j_1$ is even, we have
\[ \nu_2(j_1^{2i_1}\cdots j_s^{2i_s})\ge 2i_1, \]
and
\[ \nu_2(b_{i_1}\cdots b_{i_s}a_{i_{s+1}})
= \kappa_2(i_1)+\cdots+\kappa_2(i_{s+1})-(s+1)
\ge \nu_2(k-l)-\nu_2(i_1)-s. \]
Hence we have
\[ \nu_2(U)\ge 2i_1+\nu_2(k-l)-\nu_2(i_1)-s +\nu_2(a^{(2k)}_l). \]
If $l=0$, then we have
\[ \nu_2(U)\ge 2i_1+r-\nu_2(i_1)-s \ge r+ 2 -s  . \]
If $i\ge1 $ and $\nu_2(l)\le r$, then we have
\[ \nu_2(U)\ge 2i_1+\nu_2(l)-\nu_2(i_1)-s+(r+1-\nu_2(l))
\ge r-s+3 . \]
If $i\ge1$ and $\nu_2(l)>r$, then we have
\[ \nu_2(U)\ge 2i_1+r-\nu_2(i_1)-s \ge r+2-s. \]
This shows that $\nu_2(U)\ge r-s+2$ always holds.
From this we conclude that 
\[ \nu_2(C(j_1,\ldots,j_s))\ge r+2-s. \]
\end{proof}

Combining Corollary \ref{csodd} and Lemma \ref{cseven},
we have the following proposition.
\begin{prop}\label{propc}
If $s\ge 2$, then for any natural numbers
$j_1, j_2, \ldots, j_s$, we have
\[ \nu_2(C(j_1, j_2, \ldots, j_s))\ge r+2-2s . \]
\end{prop}

\section{Proof of the main theorem}
If $M^{4k}$ is a homotopy $\CP(2k)$, 
we obtained the relation (\ref{idx}).  
From Lemma \ref{propc}, we have
\[ \nu_2(8^{s-1}D(i_1,\ldots,i_s))\ge 3(s-1)+r+2-2s=r+s-1\ge r+1. \]
Thus from (\ref{idx}), we have
\[ \sum_{j=1}^k m_jC(j)\equiv 0 \mod 2^{r+1}. \]
From Proposition \ref{propa}, we know that
when $j$ is odd then  $C(j)\equiv 1 \mod 2^{r+1}$.
And from Proposition \ref{propa}, we have
$C(j)\equiv 0 \mod 2^{r+1}$ when $j$ is even.
Therefore we have proved that
$\sum_{1\le j\le k, j : \mbox{odd}}m_j$ is even.
This shows that  first Pontrjagin class of $\zeta$ 
\[ p_1(\zeta)=8\sum_{j=1}^k j^2 m_j \]
is divisible by 16.  This proves our main theorem.


\begin{thebibliography}{99}
\bibitem{bro}
BROWDER,~W., Surgery and the theory of differentiable
transformation groups, 
In: Proceedings of the Conference on Transformation Groups,
New Orleans 1967, Springer-Verlag, 1968, pp.~1-46.

\bibitem{br}
BRUMFIEL,~G., { Homotopy equivalences of almost 
smooth manifolds},
Comment. Math. Helv. {\bf 46} (1971), 381--407.


\bibitem{haw}
HARDY,~G.H. and WRIGHT,~E.M., 
An introduction to the theory of numbers, 
$5^{th}$ Edition, Oxford Univ. Press, 1979.


\bibitem{kit}
KITADA,~Y., On the first Pontrjagin class of homotopy complex
projective spaces, 
Math. Slovaca, 62(2012), No. 3, 551-566.

\bibitem{mjm}
MADSEN,~I., JAMES,~R. and MILGRAM,~J., The classifying spaces
for surgery and cobordism of manifolds,
Annals of Mathematics Studies, Princeton Univ. Press, 1979.

\bibitem{mat}
MASUDA,~M. and TSAI,~Y-D.,
Tangential representations of cyclic group actions on
homotopy complex projective spaces,
Osaka J. Math., 23 (1985), 907-919.

\bibitem{may}
MONTGOMERY,~D. and YANG,~C.T., 
Differentiable actions on homotopy seven spheres II,
In: Proceedings of the Conference on Transformation Groups,
New Orleans 1967, Springer-Verlag, 1968, pp.~125-134.

\bibitem{mayf}
MONTGOMERY,~D. and YANG,~C.T.,
Free differentiable actions on homotopy spheres,
In: Proceedings of the Conference on Transformation Groups,
New Orleans 1967, Springer-Verlag, 1968, pp.~175-192.

\bibitem{qui}
QUILLEN,~D., The Adams conjecture, Topology 10 (1971), 67-80.

\bibitem{san}
SANDERSON,~B.J., Immersions and embeddings of projective spaces, 
Proc. London Math. Soc., 3(1964), 137--153.

\bibitem{spi}
SPIVAK,~M., Spaces satisfying Poincar\'{e} duality,
Top. 6(1967), 77--101.



\end{thebibliography}
\end{document}